\documentclass[12pt,psfig]{article}
\usepackage{graphics,epsfig,amssymb}
\pdfoutput=1
\usepackage{cite}

\usepackage{amsmath}
\usepackage{amsthm}
\usepackage{amsfonts}
\usepackage{amssymb,latexsym}
\usepackage[all]{xy}
\usepackage{graphicx}
\usepackage[mathscr]{eucal}
\usepackage{verbatim}
\usepackage{hyperref}

\usepackage{txfonts}

\usepackage{mathrsfs}
\usepackage{color}
\usepackage{hyperref}

\theoremstyle{plain}
    \newtheorem{thm}{Theorem}[section]
    \newtheorem{prop}{Proposition}[section]
    \newtheorem{lemma}{Lemma}[section]
    \newtheorem{cor}{Corollary}[section]
    \newtheorem{defn}{Definition}[section]
    \newtheorem{rem}{Remark}[section]

\numberwithin{equation}{section}

\usepackage{graphicx}
\usepackage{pict2e}
\usepackage{slashbox}

\begin{document}

\title{Uniqueness of positive radial solutions of  Choquard type  equations \thanks{The first author was partially supported by the  Natural Science Foundation of Hunan Province (Grant No. 2022JJ30235) and the National Natural Science Foundation of China (Grant No. 12001188) ; the third author was supported by the Natural Science Foundation of Hunan Province (Grant No. 22B0484) and Natural Science Foundation of Hunan Province (Grant No. 2024JJ5214).}}

\author{
Tao Wang\\
  College of Mathematics and Computing Science,\\
Hunan University of Science and Technology, \\
Xiangtan, Hunan 411201, P. R. China\\
Email: wt\_61003@163.com;\\
Xiaoyu Tian\\
  College of Mathematics and Computing Science,\\
Hunan University of Science and Technology, \\
Xiangtan, Hunan 411201, P. R. China\\
Email: 18725613677@163.com;\\
Hui Guo\\
  College of Mathematics and Computing Science,\\
Hunan University of Science and Technology, \\
Xiangtan, Hunan 411201, P. R. China\\
Email: huiguo\_math@163.com;}
\date {}
\maketitle

\begin{abstract}
In this paper, we consider the following Choquard type equation
\begin{equation}
\left\{\begin{aligned}
&-\Delta u+\lambda u=\gamma(\Phi_N(|x|)\ast|u|^p)u  \ \ \mbox{in   $\mathbb{R}^N$}, \\
&\lim\limits_{|x|\to\infty}u(x)=0,\\
\end{aligned}\right.
\end{equation}
 where $N\geq2,\lambda>0,\gamma>0, p\in[1,2]$ and $\Phi_N(|x|)$ denotes the fundamental solution of the Laplacian $-\Delta$ on $\mathbb{R}^N$. This equation does not have a variational frame when $p\neq 2.$ Instead of variational methods,  we prove the existence  and uniqueness  of positive radial solutions of the above equation via  the  shooting method by establishing some new differential inequalities. The proofs  are based on an analysis of the corresponding system of second-order differential equations, and our results extend the existing ones in the literature from $p=2$ to $p\in[1,2]$.
\end{abstract}
\bigskip
{\bf Keywords}: {\em Nonlocal equations; Radial solutions; Uniqueness; Shooting method. }

\bigskip\noindent
{\bf Mathematics Subject Classification (2020)}: 35A02; 35B07; 35J08; 35J47.
\section{Introduction}
In this paper, we are concerned with the standing wave solutions of the following Choquard type  system
\begin{equation}\label{1.1}
\left\{\begin{aligned}
&i\psi_t+\Delta\psi-\gamma V\psi=0\quad \\
&\Delta V=|\psi|^p
\end{aligned}\right. \mbox{ in $\mathbb{R}^N\times \mathbb{R}_+$},
\end{equation}
where $i$ denotes the imaginary unit, $\gamma>0$, $N\geq2$ and $p\in[1,2]$. By using  the fundamental solution $\Phi_N$ of the Laplacian on $\mathbb{R}^N$, system \eqref{1.1} is equivalent to the nonlinear Choquard type  equation
\begin{equation}\label{1.2}
i\psi_t+\Delta\psi+\gamma(\Phi_N(|x|)\ast|\psi|^p)\psi=0\  \mbox{\ in $\mathbb{R}^N \times \mathbb{R}_+$}.
\end{equation}
In the current work, we seek the  existence and uniqueness of radial stationary solutions of the form
\begin{equation}
\psi(t,x)=u_\lambda(|x|)e^{i\lambda t},\quad u_\lambda(|x|)>0,\quad \lim\limits_{|x|\to\infty}u_\lambda(|x|)=0,
\end{equation}
where $\lambda>0$. Then equation \eqref{1.2} reduces to
\begin{equation}\label{1.4}
-\Delta u+\lambda u=\gamma(\Phi_N(|x|)\ast|u|^p)u, \mbox{\ $x\in   \mathbb{R}^N$}.
\end{equation}
Here $$\Phi_N(|x|)\ast|u|^p=\left\{
\begin{array}{lll}
&\frac{1}{N(N-2)w_N}\int_{\mathbb{R}^N}\frac{|u(y)|^p}{|x-y|^{N-2}}dy,\,\quad N\geq3,\\
&\frac{1}{2\pi}\int_{\mathbb{R}^N}\ln{\frac{1}{|x-y|}}|u(y)|^pdy,  \,\quad N=2
\end{array}\right. $$
and $w_N$ is the volume of the unit ball in $\mathbb{R}^N$.

Equation \eqref{1.4} has several physical origins.
 When $N=3,p=2$, it was first proposed in 1954 by Pekar in describing the quantum mechanics of polaron at rest. Later, in 1976, Choquard used it to model an electron trapped in its own hole, in a certain approximation to Hartree-Fock theory of one-component plasma, see \cite{Lieb}. This equation   also arises in many interesting situations related to the quantum theory of large systems of
nonrelativistic bosonic atoms and molecules  \cite{Lieb-Simon}. One can also see \cite{Palais,pen}  for more physical details.
In mathematical contents, the appearance of the convolution term makes \eqref{1.4} much
more complicated than the classical Schr\"odinger equation, and so the existence and qualitative properties of solutions for \eqref{1.4} have attracted  much attention from researchers. One can refer to \cite{Feliciangeli-Seiringer,Ma-Zhao,Moroz-VanSchaftingen,clapp,Ruiz-VanSchaftingen,Wang-Yi1,Wang-Yi2,xcl,Cassani-Tarsi,
Stubbe,Liu-Radulescu-Tang-Zhang,Guo-Wang-Yi,Liu-Radulescu-Zhang,wangli} and references therein for the existence of ground state solutions,
multiple solutions and nodal solutions to \eqref{1.4} as well as the qualitative properties such as regularity, symmetry, uniqueness and decay.
We point out that when  $p=2$, Wang and Yi \cite{Wang-Yi1} established the uniqueness of positive radial   solutions to \eqref{1.4} by developing Lieb's method in $\mathbb{R}^N$ with $N\geq3$. The existence and uniqueness of positive radially symmetric solution to  \eqref{1.4} in $\mathbb{R}^2$  have been obtained in  \cite{Cingolani-Weth}, see also \cite{Choquard-Stubbe-Vuffray}.  In view of the above works, we address the problem
\begin{itemize}
  \item whether the positive radially symmetric  solution of \eqref{1.4} is  unique up to translations when $p\neq2$?
\end{itemize}
This paper will give a confirmative answer to the case $p\in[1,2]$.

Let $r=|x|$ and $V_\lambda(r))=\Phi_N(|x|)\ast|u|^p$. Then the type $(u_\lambda(r), V_\lambda(r))\in \mathcal C^2(\mathbb{R}^N)\times C^2(\mathbb{R}^N)$ with the positive  radial solution $u_\lambda$ of \eqref{1.4}
satisfies the following corresponding  system of ordinary differential equations
\begin{equation}\label{eq1}
\left\{\begin{aligned}
&u_\lambda^{\prime\prime}+\frac{N-1}{r}u_\lambda^{\prime}=(\gamma V_\lambda+\lambda)u_\lambda,\\
&V_\lambda^{\prime\prime}+\frac{N-1}{r}V_\lambda^{\prime}=u_\lambda^{p}, r\geq0, 1\leq p\leq2.\\
\end{aligned}\right.
\end{equation}
We can assume that $u_\lambda(0)>0, V_\lambda(0)$ is finite and $u_\lambda^{\prime}(0)=V_\lambda^{\prime}(0)=0$.
By direct calculations, we obtain $V_\lambda^{\prime}(r)=\frac{1}{r^{N-1}}\int_0^ru_\lambda^{p}(s)s^{N-1}ds$,
which  implies that $V_\lambda^{\prime}\geq0$. In addition, we have

\begin{equation*}\begin{aligned}
u_\lambda^{\prime\prime}(0)=&\lim_{r\to 0}u_\lambda^{\prime\prime}(r)\\
=&\lim_{r\to 0}\frac{1-N}{r}u_\lambda^{\prime}(r)+\lim_{r\to 0}(\gamma V_\lambda(r)+\lambda)u_\lambda(r)\\
=&\lim_{r\to 0}\frac{(1-N)\int_0^r(\gamma V_\lambda(s)+\lambda)u_\lambda(s)s^{N-1}ds}{r^N}+(\gamma V_\lambda(0)+\lambda)u_\lambda(0)\\
=&\lim_{r\to 0}\frac{(1-N)((\gamma V_\lambda(r)+\lambda)u_\lambda(r)}{N}+(\gamma V_\lambda(0)+\lambda)u_\lambda(0)\\
=&\frac{(\gamma V_\lambda(0)+\lambda)u_\lambda(0)}{N}.
\end{aligned}\end{equation*}
This combined with the fact that $\lim\limits_{|x|\to\infty}u_\lambda(|x|)=0$, implies that
$\lambda+\gamma V_\lambda(0)<0$.
By taking $u(r)=Au_\lambda(\frac{r}{\sigma})$, $V(r)=B[V_\lambda(\frac{r}{\sigma})-V_\lambda(0)]$ with
$$\sigma^{2}=-\lambda-\gamma V_\lambda(0),\quad B=\frac{\gamma}{\sigma^2},\quad A=(\frac{B}{\sigma^2})^\frac{1}{p},$$
we have $$u_\lambda(r)=\frac{1}{A}u(\sigma r),\quad V_\lambda(r)=\frac{1}{B}V(\sigma r)+V_\lambda(0).$$ Then we obtain $$u_\lambda^{\prime}=\frac{\sigma}{A}u^{\prime}(\sigma r), u_\lambda^{\prime\prime}=\frac{\sigma^2}{A}u^{\prime\prime}(\sigma r), V_\lambda^{\prime}=\frac{\sigma}{B}V^{\prime}(\sigma r), V_\lambda^{\prime\prime}=\frac{\sigma^2}{B}V^{\prime\prime}(\sigma r).$$
Inserting these variables  into \eqref{eq1},   we obtain the  following  canonical form
\begin{equation}\label{jihao1}
\left\{\begin{aligned}
&u^{\prime\prime}+\frac{N-1}{r}u^{\prime}=(V-1)u,\\
&V^{\prime\prime}+\frac{N-1}{r}V^{\prime}=u^{p}, 1\leq p\leq2,\\
\end{aligned}\right.
\end{equation}
subject to the initial conditions
\begin{equation}\label{jihao2}
u(0)=u_0\in\mathbb{R}_+, u^{\prime}(0)=0, V(0)=0, V^{\prime}(0)=0.
\end{equation}
By applying similar spirits as in \cite{Choquard-Marc} with necessary modifications, we shall investigate system \eqref{jihao1} under  initial-value conditions \eqref{jihao2} by using shooting method. This method plays a key role in the proof of the uniqueness results, which is started from the the work of Coffman \cite{Coffman1} back to 1972 to study
the following  semilinear elliptic differential equation
\begin{equation}\label{fu}
\Delta u-u+u^3=0, x\in{\mathbb{R}}^N.
\end{equation}
Decades later, it was further developed after  contributions of many authors in the entire space, among whom we mention Kwong \cite{Kwong}, Coffman \cite{Coffman1,Coffman2}, Peletier and Serrin \cite{Peletier-Serrin}, Mc Leod \cite{McLeod}, Tang \cite{Tang1}, Felmer \cite{Felmer-M-T} et al. There have also been many extensions and refinements of this method and we cannot quote them all, but we would like to mention the works by Clemons and Jones \cite{Clemons-Jones}, Erbe and Tang \cite{Erbe-Tang}, Kabeya and Tanaka \cite{Kabeya-Tanaka}.

Our main result is as follows.

\begin{thm}\label{th}
For any $N\geq2$, equation \eqref{1.4} with $\gamma>0$ and $\lambda>0$ admits a unique radially symmetric solution such that $u(|x|)>0$ and
\begin{equation}\label{jihao}
\lim\limits_{|x|\to\infty}u(|x|)=0.
\end{equation}
\end{thm}
In order to prove Theorem \ref{th}, we need to verify the uniqueness of positive solutions of \eqref{jihao1} with initial-values \eqref{jihao2}. For the initial condition $u_0:=u(0)>0$ of the solution $(u, V)$ to \eqref{jihao1}, we introduce the following three subsets in $\mathbb{R}_+:=(0,\infty)$.
\begin{defn}\label{DY2.1}
\begin{equation}\label{jihao4}
\mathcal{N}=\{u_0\in\mathbb{R}_+:\exists r_0>0 \mbox{ such that }u(r_0)<0 \mbox{ and } u^{\prime}(r)<0\mbox{ on }(0,r_0]\},
\end{equation}
\begin{equation}\label{jihao5}
\mathcal{G}=\{u_0\in\mathbb{R}_+:u\geq0, \lim_{r\to\infty}u(r)=0\},
\end{equation}
\begin{equation}\label{jihao6}
\mathcal{P}=\{u_0\in\mathbb{R}_+:\exists r_1>0 \mbox{ such that }u^{\prime}(r_1)>0 \mbox{ and } u(r)>0\mbox{ on }(0,r_1]\}.
\end{equation}
\end{defn}
\noindent The main difficulties in the proof of \eqref{jihao1} with initial-values \eqref{jihao2} is to obtain the nonemptiness of $\mathcal{N}, \mathcal{P}$ and $ \mathcal{G}$ and the decay of $u$, which shall be overcame by constructing new auxiliary functions and establishing subtle differential inequalities.

The outline of the paper is organized as follows. In Section 2, firstly, we prove the nonemptyness of the subsets $\mathcal{N}, \mathcal{P}$ and $\mathcal{G}$, and then presents their properties. Finally, we obtain the  existence of positive radially symmetric solutions to \eqref{1.4}.
In Section 3, we give the decay estimates and  then establish the  uniqueness of the positive radial solutions by analyzing the Wronskian of solutions of \eqref{jihao1}.
\section{Existence of positive radial solutions}
In this section, we are devoted to the proof of the existence of positive radial solutions  by discussing some general properties of solutions of \eqref{jihao1} with initial values \eqref{jihao2}.

First we give some preliminaries. In view of \eqref{jihao1} and \eqref{jihao2}, we have
\begin{equation}\label{jihao3}
\left\{\begin{aligned}
&u^{\prime}(r)=\frac{1}{r^{N-1}}\int_0^r(V(s)-1)u(s)s^{N-1}ds,\\
&V^{\prime}(r)=\frac{1}{r^{N-1}}\int_0^ru^{p}(s)s^{N-1}ds.\\
\end{aligned}\right.
\end{equation}
Next, we illustrate  the relationship of $\mathcal{N}, \mathcal{P}$ and $\mathcal{G}$.
\begin{lemma}\label{lem1}
$\mathcal{G}$ and $\mathcal{P}$ are disjoint.
\end{lemma}
\begin{proof}
We claim that for any $u_0\in \mathcal{P}$, the corresponding solution $u$ satisfies
\begin{equation}\label{duanyan1}
\lim\limits_{r\to\infty}u(r)>0.
\end{equation}
In fact, observe that $u(0)=u_0>0, u^{\prime}(0)=0$ and
\begin{equation*}\begin{aligned}
u^{\prime\prime}(0)=&\lim_{r\to 0}u^{\prime\prime}(r)\\
=&\lim_{r\to 0}\frac{1-N}{r}u^{\prime}(r)+\lim_{r\to 0}(V(r)-1)u(r)\\
=&\lim_{r\to 0}\frac{(1-N)\int_0^r(V(s)-1)u(s)s^{N-1}ds}{r^N}-u_0\\
=&\lim_{r\to 0}\frac{(1-N)(V(r)-1)u(r)}{N}-u_0\\
=&-\frac{u_0}{N}<0.
\end{aligned}\end{equation*}
Then all solutions start strictly decreasing from $0$. If $u_0\in\mathcal{P}$, then $u$ must have a local minimum point ${r}_{\min}\in(0,r_1)$ such that
$$u({r}_{\min})>0, u^{\prime}({r}_{\min})=0\mbox{ and }u^{\prime\prime}({r}_{\min})\geq 0.$$
Inserting it into \eqref{jihao1}, we get $V({r}_{\min})\geq1$. Furthermore, since $V$ is a strictly increasing function due to \eqref{jihao3}, we deduce that $V(r)>V({r}_{\min})\geq1$.
By integrating \eqref{jihao1} over $[r_{\min},r]$, we deduce  $$u^{\prime}(r)=\frac{1}{r^{N-1}}\int_{r_{min}}^r(V(s)-1)u(s)s^{N-1}ds>0\mbox{ on }(r_{\min},\infty)$$ and $u(r)>u(r_{\min})>0$ for any $r>r_{\min}$. So the claim \eqref{duanyan1} holds.

Thus $\mathcal{G}$ and $\mathcal{P}$ are disjoint due to the fact that $\lim\limits_{r\to\infty}u(r)=0$ for any $u_0\in\mathcal{G}$. So the proof is completed.
\end{proof}

\begin{rem}\label{remark1}
According to the definition of $\mathcal{N}, \mathcal{G}, \mathcal{P}$ and Lemma \ref{lem1}, we see $\mathcal{N}\cap\mathcal{P}={\o}, \mathcal{N}\cap\mathcal{G}={\o}$ and $\mathcal{G}\cap\mathcal{P}={\o}$, that is, $\mathcal{N},\mathcal{P},\mathcal{G}$ are mutually disjoint sets.
\end{rem}

\begin{lemma}\label{lem2}
$\mathcal{N}\cup\mathcal{G}\cup\mathcal{P}=\mathbb{R}_+.$
\end{lemma}
\begin{proof}
We define
\begin{equation}\label{N1}
\mathcal{N}_1=\{u_0\in\mathbb{R}_+:\exists r_0>0 \mbox{ such that }u(r_0)<0 \mbox{ and } u^{\prime}(r)\leq0\mbox{ on }(0,r_0]\},
\end{equation}
\begin{equation}\label{P1}
\mathcal{P}_1=\{u_0\in\mathbb{R}_+:\exists r_1>0 \mbox{ such that }u^{\prime}(r_1)>0 \mbox{ and } u(r)\geq0\mbox{ on }(0,r_1]\}.
\end{equation}
It is easy to know that $\mathcal{N}_1\cup\mathcal{G}\cup\mathcal{P}_1=\mathbb{R}_+$ from the definition of $\mathcal{N}_1, \mathcal{P}_1, \mathcal{G}$. We claim
\begin{equation}\label{NGP} \mathcal{N}\cup\mathcal{G}\cup\mathcal{P}=\mathcal{N}_1\cup\mathcal{G}\cup\mathcal{P}_1=\mathbb{R}_+.
\end{equation}

In order to prove \eqref{NGP},  we first claim $\mathcal{N}_1=\mathcal{N}.$
We just have to prove that $\mathcal{N}_1\backslash\mathcal{N}={\o}$. Otherwise,
for any $u_0\in\mathcal{N}_1\backslash\mathcal{N}$, there exists $r_{01}\in(0,r_0)$ such that $u^{\prime}(r_{01})=0$ and $u^{\prime}(r)<0$ on $(0,r_{01})$.
If $u(r_{01})>0$, then we have $u^{\prime\prime}(r_{01})\geq0$ because
$$u^{\prime\prime}(r_{01})=u_{-}^{\prime\prime}(r_{01})=\lim\limits_{r\to r_{01}^-}\frac{u^{\prime}(r)-u^{\prime}(r_{01})}{r-r_{01}}\geq0.$$
Inserting it into \eqref{jihao1}, we get $V(r_{01})\geq1$. Furthermore, $V(r)>V(r_{01})\geq1$ on $(r_{01},\infty)$. Finally, by integrating \eqref{jihao1} over $[r_{01},r]$,  we obtain $$u^{\prime}(r)=\frac{1}{r^{N-1}}\int_{r_{01}}^r(V(s)-1)u(s)s^{N-1}ds>0\mbox{ on}(r_{01},\infty),$$ which contradicts with  $u(r_0)<0$. If $u(r_{01})=0$,  we suppose that $V(r_{01})=V_{01}$ and $V^{\prime}(r_{01})=\bar V_{01}$. Then $(0,V_{01}+\frac{\bar V_{01}}{2-N}r_{01}^{N-1}r^{2-N}-\frac{V_{01}}{2-N}r_{01})$ is the unique solution of
\begin{equation}\label{xin}
\left\{\begin{aligned}
&u^{\prime\prime}+\frac{N-1}{r}u^{\prime}=(V-1)u,\\
&V^{\prime\prime}+\frac{N-1}{r}V^{\prime}=u^{p},\\
&u(r_{01})=u^{\prime}(r_{01})=0, V(r_{01})=V_{01}, V^{\prime}(r_{01})=\bar V_{01},\\
\end{aligned}\right.
\end{equation}
which contradicts with  $u(r_0)<0$ according to \eqref{N1}. Then $\mathcal{N}_1\backslash\mathcal{N}={\o}$. That is to say $\mathcal{N}_1=\mathcal{N}$.\\
\indent Next we claim $\mathcal{P}_1=\mathcal{P}.$
We just have to prove that $\mathcal{P}_1\backslash\mathcal{P}={\o}$. Otherwise,
for any $u_1\in\mathcal{P}_1\backslash\mathcal{P}$, there exists $r_{02}\in(0,r_1)$ such that $u(r_{02})=0$, which is the minimum point and thereby $u^{\prime}(r_{02})=0$. By using similar arguments as in \eqref{xin}, it follows that there is only the zero solution due to $u(r_{02})=u^{\prime}(r_{02})=0$, which leads to a contradiction with $u^{\prime}(r_1)>0$ because $r_1>r_{02}$. The claim holds.

From the above statement, we obtain $$\mathcal{N}\cup\mathcal{G}\cup\mathcal{P}=\mathcal{N}_1\cup\mathcal{G}\cup\mathcal{P}_1=\mathbb{R}_+.$$
Therefore, the proof is completed.
\end{proof}

\begin{lemma}\label{lem7}
$\mathcal{N}, \mathcal{P}$ are open sets.
\end{lemma}
\begin{proof}
Let $u_0\in\mathcal{N}$. The corresponding solution $u$ takes a negative value at some point $r=r_0$. By  the continuity theorem of solution on initial conditions, there exists $\delta$ such that $B_\delta(u_0)\subset\mathcal{N}$ for each $u_0\in\mathcal{N}$. So $\mathcal{N}$ is an open set.
Let $u_0\in\mathcal{P}$. Then $u^{\prime}(r_0)>0$ for some point $r=r_0$.  It follows from the continuity theorem of solution on initial conditions again that there exists $\varepsilon$ such that $B_\varepsilon(u_0)\subset\mathcal{P}$ for each $u_0\in\mathcal{P}$.   So we obtain $\mathcal{P}$ is an open set.
\end{proof}

In the following, we are ready to prove the nonemptiness of the three subsets $\mathcal{N},\mathcal{G},\mathcal{P}$, which is a novel point in this work.
\begin{lemma}\label{lem3}
The set $\mathcal{G}$ is non-empty and there is a solution $(u,V)$  of system \eqref{jihao1} subject to the initial conditions \eqref{jihao2} such that $u(r)>0,u^{\prime}(r)<0$ on $(0,\infty)$ and $\lim\limits_{r\to\infty}u(r)=0$.
\end{lemma}
\begin{proof}
According to Lemmas \ref{lem2} and  \ref{lem7}, we can deduce that $\mathcal{G}$ is nonempty. Furthermore, we claim
$$u(r)>0\mbox{ on }(0,\infty).$$
In fact, suppose on the contrary that there exists $\bar{r}>0$ such that $u(r)>0$ in $(0,\bar{r})$ and $u(\bar{r})=0$, which is the minimum point in $\mathbb{R}_+$ and thereby $u^{\prime}(\bar{r})=0$. Then we suppose that $V(\bar{r})=\bar{V}$ and $V^{\prime}(\bar{r})=\bar{V}^{\prime}$. Then $(0,\bar{V}+\frac{\bar{V}^{\prime}}{2-N}\bar{r}^{N-1}r^{2-N}-\frac{\bar{V}}{2-N}\bar{r})$ is the unique solution of
\begin{equation*}
\left\{\begin{aligned}
&u^{\prime\prime}+\frac{N-1}{r}u^{\prime}=(V-1)u,\\
&V^{\prime\prime}+\frac{N-1}{r}V^{\prime}=u^{p},\\
&u(\bar{r})=u^{\prime}(\bar{r})=0, V(\bar{r})=\bar{V}, V^{\prime}(\bar{r})=\bar{V}^{\prime}.\\
\end{aligned}\right.
\end{equation*}
By the existence and uniqueness of the solution, it follows that there is only the zero solution duo to $u(\bar{r})=u^{\prime}(\bar{r})=0$, which is a contradiction with $u_0\neq0$. So the claim is proved.

Next we claim
$$u^{\prime}(r)<0\mbox{ on }(0,\infty).$$
Suppose that there exist the first critical point $r_1$ such that $u^{\prime}(r_1)=0$. It is easy to know that $u^{\prime\prime}(r_1)\geq0$. This together with \eqref{jihao1},  yields that  $V(r_1)\geq1.$ Furthermore, $V(r)>V(r_1)\geq1$ for $r>r_1$. Then by integrating \eqref{jihao1} over $[r_1,r]$, we get $$u^{\prime}(r)=\frac{1}{N-1}\int_{r_1}^r(V(s)-1)u(s)s^{N-1}ds>0\mbox{ on }(r_1,\infty).$$ We obtain a contradiction  with $u(\infty)=0$. So the claim is proved.

Thus if $u_0\in\mathcal{G}$, then $u(r)>0,u^{\prime}(r)<0$ on $(0,\infty)$ and $\lim\limits_{r\to\infty}u(r)=0$.
\end{proof}

\begin{lemma}\label{lem4}
The sets $\mathcal{N}$ is nonempty. In particular, $(0,\frac{1}{4})\subset\mathcal{N}.$
\end{lemma}
\begin{proof}
We prove it by contradiction. For any $u_0\in(0,\frac{1}{4})$ and suppose $u_0\notin\mathcal{N}$, then $u_0\in\mathcal{G}\cup\mathcal{P}$.  Then we claim $u>0\mbox{ on }(0,\infty).$
In fact, if $u_0\in\mathcal{G}$, then  we have $u>0$ by using Lemma \ref{lem3}. If $u_0\in\mathcal{P}$, there exists $r_{\min}$ such that $u(r_{\min})>0, u^{\prime}(r_{\min})=0$ and $u^{\prime\prime}(r_{\min})\geq0$. Inserting it into \eqref{jihao1}, we get $V(r_{\min})\geq1$. Furthermore, $V(r)>V(r_{\min})\geq1$. Then by integrating \eqref{jihao1} over $[r_{\min},r]$, we obtain $$u^{\prime}(r)=\frac{1}{N-1}\int_{r_{\min}}^r(V(s)-1)u(s)s^{N-1}ds>0\mbox{ on }(r_{\min},\infty).$$ Thus $u(r)>u(r_{\min})>0$. So the claim is proved.

Now we consider the function
\begin{equation}\label{jihao7}
\phi=2u+V-\frac{1}{2}.
\end{equation}
Since $p\in[1,2]$,   $\phi$ satisfies the differential inequality equation
\begin{equation}\label{jihao11}
\phi^{\prime\prime}+\frac{N-1}r\phi^{\prime}=2(V-1)u+u^{p}=2\phi u-4u^2-u+u^{p}<2\phi u.
\end{equation}
Obviously, $\phi(0)=2u_0-\frac{1}{2}=2(u_0-\frac{1}{4})<0, \phi^{\prime}(0)=0$ and
\begin{equation*}\begin{aligned}
\phi^{\prime\prime}(0)=&2u^{\prime\prime}(0)+V^{\prime\prime}(0)\\
=&2\lim_{r\to 0}u^{\prime\prime}(r)+\lim_{r\to 0}V^{\prime\prime}(r)\\
=&2\lim_{r\to 0}\frac{1-N}{r}u^{\prime}(r)+2\lim_{r\to 0}(V(r)-1)u(r)+\lim_{r\to 0}\frac{1-N}{r}V^{\prime}(r)+\lim_{r\to 0}u^p(r)\\
=&2\lim_{r\to 0}\frac{(1-N)\int_0^r(V(s)-1)u(s)s^{N-1}ds}{r^N}-2u_0+\lim_{r\to 0}\frac{(1-N)\int_0^ru^p(s)s^{N-1}ds}{r^N}+u_0^p\\
=&2\lim_{r\to 0}\frac{(1-N)(V(r)-1)u(r)}{N}-2u_0+\lim_{r\to 0}\frac{(1-N)u^p(r)}{N}+u_0^p\\
=&-2\frac{u_0}{N}+\frac{u_0^p}{N}\\
=&\frac{u_0}{N}(u_0^{p-1}-2)<0.
\end{aligned}\end{equation*}
Then all solutions start strictly decreasing from 0.\\
\indent We claim that $\phi$ is decreasing on $(0,\infty)$. Indeed, if it's not true, there must be a local minimum point
$r_\ast>0$ such that
\begin{equation}\label{jihao12}
\phi(r_\ast)<\phi(0)<0, \phi^{\prime}(r_\ast)=0, \phi^{\prime\prime}(r_\ast)\geq0.
\end{equation}
Inserting $\phi(r_\ast)<0\mbox{ and }\phi^{\prime}(r_\ast)=0$ into \eqref{jihao11}, we get $$\phi^{\prime\prime}(r_\ast)<2\phi(r_\ast)u(r_\ast)<0,$$which leads to a contradiction with \eqref{jihao12}.
So the claim is proved.

By the claim, we deduce from \eqref{jihao7} that $2u(r)+V(r)\leq2u_0$ on $(0,\infty)$.
Obviously,
\begin{equation}\label{deduce2}
V(r)\leq2u_0<1\mbox{ on }(0,\infty).
\end{equation}
Hence $V(r)$ is bounded. Since $V$ is always strictly increasing,  $\lim\limits_{r\to\infty}V(r)= :V_\infty$ exists and $V_\infty\leq2u_0<1$.
We consider the function $$z=-\frac{u^{\prime}}{u}.$$ It satisfies the differential equation
\begin{equation}\label{z}
z^{\prime}=z^2-\frac{N-1}rz+1-V.
\end{equation}

If $u_0\in\mathcal{G}$,  it follows that $z$ exists for all $r>0$ and $z(r)>0$ for all $r>0$.
Choose $\tilde{r}$ such that $\frac{N-1}{r}\leq\frac{N-1}{\tilde{r}}\leq\sqrt{2(1-V_\infty)}$ for all $r\geq\tilde{r}$. Then we have
\begin{equation*}\begin{aligned}
z^{\prime}\geq&\frac{1}{2}z^2+(\frac{1}{2}z^2-\sqrt{2(1-V_\infty)}z+1-V_\infty)\\
=&\frac{1}{2}z^2+\frac{1}{2}(z-\sqrt{2(1-V_\infty)})^2\\
\geq&\frac{1}{2}z^2.
\end{aligned}\end{equation*}
on $(\tilde{r},\infty)$.
Thus we have $$\frac{z^{\prime}}{z^2}\geq\frac{1}{2}\mbox{ on }(\tilde{r},\infty).$$ Furthermore, we have
\begin{equation}\label{333}
(\frac{1}{z}+\frac{r}{2})^{\prime}\leq0\mbox{ on }(\tilde{r},\infty).
\end{equation}
Integrating \eqref{333} over$(\tilde{r},r)$, we obtain $$\int_{\tilde{r}}^r(\frac{1}{z}+\frac{r}{2})^{\prime}\leq0.$$
So we see $$\frac{1}{z(r)}\leq\frac{1}{z(\tilde{r})}+\frac{1}{2}(\tilde{r}-r)\mbox{ on }(\tilde{r},\infty).$$
This shows that $z(r)$ is negative at some $r$ large enough, which contradicts with $z(r)>0$. Thus $u_0\notin \mathcal{G}$.

If $u_0\in\mathcal{P}$, from the above statement we have $V(r)>1$ on $(r_{\min},\infty)$, which contradicts with \eqref{deduce2}. Thus $u_0\notin \mathcal{P}$. The above arguments yield a contradiction with Lemma \ref{lem2}.

Hence, $u_0\in\mathcal{N}$. In particular, $(0,\frac{1}{4})\subset\mathcal{N}$.
\end{proof}

\begin{lemma}\label{lem5}
For $u_0$ large enough, $u_0\in\mathcal{P}$.
\end{lemma}
\begin{proof}
Suppose on the contrary that $\mathcal{P}$ is empty for $u_0$ large enough. Let
\begin{equation}\label{R0}
R_0=\sup\{\bar r>0, u(r)>0, u^{\prime}(r)<0 \mbox{  for any  } r\in(0,\bar r]\}.
\end{equation}
We consider the function
\begin{equation}\label{jihao10}
\phi_2(r)=u+\lambda_0V-\lambda_0, \lambda_0=u_0^{\frac{2-p}{2}}, r\in[0,R_0).
\end{equation}
It satisfies the differential inequality equation
\begin{equation}\label{phi2}
\phi_2^{\prime\prime}+\frac{N-1}{r}\phi_2^{\prime}=(V-1)u+\lambda_0u^{p}=\frac{1}{\lambda_0}\phi_2u+\lambda_0u^{p}-\frac{1}{\lambda_0}u^2+u>\frac{1}{\lambda_0}\phi_2u
\end{equation}
for $p\in[1,2]$. When $u_0$ is large enough, we have
$$\phi_2(0)=u_0-\lambda_0>0, \phi_2^{\prime}(0)=0, \phi_2^{\prime\prime}(0)
=u^{\prime\prime}(0)+\lambda_0V^{\prime\prime}(0)=-\frac{u_0}{N}+\frac{\lambda_0u_0^{p}}{N}>0.$$
We claim $\phi_2$ is strictly increasing. Otherwise there must exist a maximum point $r^{\prime}$ such that $$\phi_2(r^{\prime})>0, \phi_2^{\prime}(r^{\prime})=0, \phi_2^{\prime\prime}(r^{\prime})\leq0.$$ However, according to \eqref{phi2},  we obtain
$$\phi_2^{\prime\prime}(r^{\prime})>\frac{1}{\lambda_0}\phi_2(r^{\prime})u(r^{\prime})>0.$$
This leads to a contradiction. The claim holds  and  $\phi_2$ is strictly increasing.
Thus $\phi_2(r)>\phi_2(0)=u_0-\lambda_0$. Furthermore,
\begin{equation}\label{jihao15}
u(r)>u_0-\lambda_0V(r)\mbox{ on } (0,R_0).
\end{equation}
On the other hand,  we conclude from equation \eqref{jihao3} for $V^{\prime}$ that
$$\frac{u^{p}r}{N}\leq V^{\prime}\leq\frac{u_0^{p}r}{N}\mbox{ on }(0,R_0).$$
Integrating these inequalities and using again the fact that $u$ is decreasing, we have the following estimates for $V$ :
\begin{equation}\label{vfanwei}
\frac{u^{p}r^2}{2N}\leq V\leq\frac{u_0^{p}r^2}{2N}\mbox{ on }(0,R_0).
\end{equation}
This together with \eqref{jihao10} and \eqref{jihao15} gives
\begin{equation}\label{1111}
u(r)> u_0-\lambda_0\frac{u_0^{p}r^2}{2N}=u_0-\frac{u_0^{1+\frac{p}{2}}r^2}{2N}\mbox{ on }(0,R_0).
\end{equation}
Let $$g(r)=u_0-\frac{u_0^{1+\frac{p}{2}}r^2}{2N}.$$ It's clear that $g(r)$ is strictly decreasing. There exists $r_0>0$ such that $g(r_0)=0$. Furthermore, we have $$r_0=\sqrt{2N/u_0^{\frac{p}{2}}}.$$
Let $r*=\min\{r_0,R_0\}$. Then according to \eqref{1111},  we have
\begin{equation}\label{2222}
u(r)> u_0-\frac{u_0^{1+\frac{p}{2}}r^2}{2N}=u_0(1-\frac{r^2}{r_0^2})>0\mbox{ on }(0,r*).
\end{equation}
If $r_0\leq R_0$, let $r_1=\frac{1}{2}r_0<r_0$.
Then
\begin{equation*}\begin{aligned}
u^{\prime}(r_1)=&\frac{1}{r_1^{N-1}}\int_0^{r_1}(V(r)-1)u(r)r^{N-1}dr\\
\geq&\frac{1}{r_1^{N-1}}\int_0^{r_1}\frac{u^{p+1}(r)r^{N+1}}{2N}dr-\frac{1}{r_1^{N-1}}\int_0^{r_1}u(r)r^{N-1}dr\\
\geq&\frac{u_0^{p+1}}{2Nr_1^{N-1}}\int_0^{r_1}(1-\frac{r^2}{r_0^2})^{p+1}r^{N+1}dr-\frac{u_0r_1}{N}\\
\geq&\frac{u_0^{p+1}r_0^{N+2}}{2Nr_1^{N-1}}\int_0^{\frac{r_1}{r_0}}(1-s^2)^{p+1}s^{N+1}ds-\frac{u_0r_1}{N}\\
=&\frac{u_0r_1}{N}[\frac{u_0^pr_0^N}{2r_1^N}\frac{2N}{u_0^{\frac{p}{2}}}\int_0^{\frac{r_1}{r_0}}(1-s^2)^{p+1}s^{N+1}ds-1]\\
=&\frac{u_0r_1}{N}[Nu_0^{\frac{p}{2}}(\frac{r_0}{r_1})^N\int_0^{\frac{r_1}{r_0}}(1-s^2)^{p+1}s^{N+1}ds-1]\\
=&\frac{u_0r_0}{2N}[2^NNu_0^{\frac{p}{2}}\int_0^{\frac{1}{2}}(1-s^2)^{p+1}s^{N+1}ds-1].
\end{aligned}\end{equation*}
We conclude that $u^{\prime}(r_1)>0$ for $u_0$ sufficiently large, which contradicts with the definition of $R_0$. Thus $r_0>R_0$. Let $$\delta=g(R_0)=u_0-\frac{u_0^{1+\frac{p}{2}}R_0^2}{2N}.$$ Then we have $\delta=g(R_0)>g(r_0)=0$ since $g$ is decreasing. For any $r\in(0,R_0)\varsubsetneqq(0,r_0)$, we have
$$u(r)>u_0-\frac{u_0^{1+\frac{p}{2}}r^2}{2N}\geq\delta>0\mbox{ on }(0,r_0)$$
according to \eqref{1111}.
Furthermore, we get $u(R_0)\geq\delta>0$. According to the definition of $R_0$, then we conclude $$u^{\prime}(R_0)=0.$$
Inserting it to \eqref{jihao3}, we see $V(R_0)>1$. Furthermore, according to \eqref{jihao1}, we get $u^{\prime\prime}(R_0)>0$. So $R_0$ must be the local minimum point. Due to the continuity of $u^{\prime}$ and $u$, there exist $\varepsilon>0$ such that $u^{\prime}(R_0+\varepsilon)>0$ and $u(R_0+\varepsilon)>0$.
So we have $u^{\prime}(R_0+\varepsilon)>0$ and $u(r)>0$ on $(0,R_0+\varepsilon)$, which contradicts the assumption that $\mathcal{P}$ is empty.

Thus $u_0\in\mathcal{P}$ for $u_0$ large enough. The proof is finished.
\end{proof}
\section{Uniqueness of positive radial solutions}
In this section, we prove that $\mathcal{G}$ has exactly one element.  Firstly, we illustrate the following lemma  which  shows that  any two solutions of  \eqref{jihao1} with the initial values \eqref{jihao2} cannot intersect as long as they are positive by analysing the features of their Wronskian.

\begin{lemma}\label{lem6}
Let $u_0\in\mathcal{G}\cup\mathcal{P}$. Suppose that $u_2(0)>u_1(0)>0$ and $u_2(r),u_1(r)$ exist on $[0,\infty)$ such that $u_1(r)\geq0$ on $[0,\infty)$. Then $u_2(r)>u_1(r)$ for any $r\in[0,\infty)$.
\end{lemma}
\begin{proof}
Let $\omega(r)$ be the Wronskian of $u_1, u_2$ defined by
\begin{equation}\label{jihao8}
\omega(r)=u_2^{\prime}(r)u_1(r)-u_1^{\prime}(r)u_2(r).
\end{equation}
Then $\omega$ satisfies the differential equation
\begin{equation}\label{jihao9}
\omega^{\prime}+\frac{N-1}{r}\omega=(V_2-V_1)u_1u_2.
\end{equation}
We claim $u_2(r)>u_1(r)$ in $[0,\infty)$. In fact, suppose on the contrary that there exists $\bar{r}\in(0,\infty)$ such that
\begin{equation}\label{jiashe}
u_2(\bar{r})=u_1(\bar{r})\mbox{ and }u_2(r)>u_1(r)\mbox{ in }[0,\bar{r}).
\end{equation}
Note that $$(V_2(r)-V_1(r))^{\prime}=\frac{1}{r^{N-1}}\int_0^r(u_2^p(s)-u_1^p(s))s^{N-1}ds>0\mbox{ on } (0,\bar{r}].$$
Since $V_2(0)-V_1(0)=0$, we deduce that
$$V_2(r)>V_1(r)\mbox{ on }r\in(0,\bar{r}].$$
From the differential equation \eqref{jihao9} that we have
$$(\omega r^{N-1})^{\prime}=r^{N-1}(V_2-V_1)u_1u_2>0\mbox{ on }(0,\bar r].$$
We conclude that $\omega r^{N-1}$ is increasing on $(0,\bar{r})$. This combined with the fact $\omega(0)=0$,
\begin{equation}\label{wbar}
\omega(\bar{r})>0.
\end{equation}
On the other hand, set $$F(r)=u_2(r)-u_1(r).$$ Then $F(r)>0$ on $[0,\bar{r})$ and $F(\bar{r})=0$ due to \eqref{jiashe}. So $$u_2^{\prime}(\bar{r})-u_1^{\prime}(\bar{r})=F^{\prime}(\bar{r})=\lim\limits_{r\to\ \bar{r}}\frac{F(r)-F(\bar{r})}{r-\bar{r}}\leq0.$$
It is easy to know that $$\omega(\bar{r})=u_1(\bar{r})(u_2^{\prime}(\bar{r})-u_1^{\prime}(\bar{r}))\leq0,$$ which leads to a contradiction with \eqref{wbar}. Thus the claim is proved.
\end{proof}

\begin{rem}\label{remark3}
If $u_0\in\mathcal{N}$. Suppose that $u_2(0)>u_1(0)>0$ and $u_2(r),u_1(r)$ exist on $[0,R]$ such that $u_1(r)\geq0$ on $[0,R]$. It is then easy to prove $u_2(r)>u_1(r)$ for any $r\in[0,R]$ by imitating the proof of the Lemma \ref{lem6}.
\end{rem}
In the sequel, we need to study the decay estimate of the solutions in $\mathcal{G}$ at infinity.
Since $V$ is always strictly increasing, we deduce that $V_\infty :=\lim\limits_{r\to\infty}V(r)$ exists (including the case $V_\infty=+\infty$) and $V(r)<V_\infty$ for all $r>0$.
Then we have the following conclusion.
\begin{prop}\label{prop1}
Let $u_0\in\mathcal{G}$. Then $V_\infty\geq1$.
\end{prop}
\begin{proof}
First of all, let $V_\infty<\infty$. We consider the function $z :=-\frac{u^{\prime}}{u}$ which is well defined for all $r\geq0$ and $z$ satisfies the differential equation
$$z^{\prime}=-\frac{u^{\prime\prime}u-(u^{\prime})^2}{u^2}=z^2-\frac{N-1}{r}z+1-V.$$
Now, choose $\tilde{r}$ such that $\frac{N-1}{r}\leq\frac{N-1}{\tilde{r}}\leq\frac{1}{2}\sqrt{V_\infty}$ for all $r\geq \tilde{r}$. Consider the direction field in the $(r,z)$ plane for the preceding differential equation. There holds that
\begin{equation*}\begin{aligned}
z^{\prime}\geq&\frac{1}{2}z^2+(\frac{1}{2}z^2-\frac{1}{2}\sqrt{V_\infty}z+1-V_\infty)\\
\geq&\frac{1}{2}z^2+(\frac{1}{2}z^2-\frac{1}{2}\sqrt{V_\infty}z-V_\infty)+1\\
\geq&\frac{1}{2}z^2+(\frac{1}{2}z-\sqrt{V_\infty})(z+\sqrt{V_\infty})+1\\
\geq&\frac{1}{2}z^2+1\\
\geq&\frac{1}{2}z^2
\end{aligned}\end{equation*}
when $$r\geq \tilde{r}, z\geq2\sqrt{V_\infty}.$$  By applying  similar arguments as in the proof of Lemma \ref{lem4},  we see that the function $z$ satisfies
\begin{equation}\label{zzz}
\frac{1}{z(r)}\leq\frac{1}{z(\tilde{r})}+\frac{1}{2}(\tilde{r}-r)\mbox{ on }(\tilde{r},\infty).
\end{equation}
Let $$h(r)=\frac{1}{z(\tilde{r})}+\frac{1}{2}(\tilde{r}-r)\mbox{ on }(\tilde{r},\infty).$$
Suppose that there exists $r_1>\tilde{r}$ such that $h(r_1)=0$. Then we have $h(r)<0$ for any $r>r_1$, which yields a contradiction with the definition fo $z(r)$. Hence  $z(r)$ remains bounded.

Since $z(r)=-\frac{u^{\prime}(r)}{u(r)}$, we have $u^{\prime}(r)=-z(r)u(r)$, and so $$\lim\limits_{r\to\infty}u^{\prime}(r)=-\lim\limits_{r\to\infty}z(r)u(r)=0.$$
Therefore, we conclude from  l'H\^{o}spital's ruler that
\begin{equation}\begin{aligned}\label{limz}
\lim\limits_{r\to\infty}z^2=&\lim\limits_{r\to\infty}\frac{{u^{\prime}}^2}{u^2}\\
=&\lim\limits_{r\to\infty}\frac{u^{\prime\prime}}{u^{\prime}}\\
=&\lim\limits_{r\to\infty}\frac{(V-1)u-\frac{N-1}{r}u^{\prime}}{u}\\
=&\lim\limits_{r\to\infty}(\frac{N-1}{r}z+V-1)\\
=&V_\infty-1.
\end{aligned}\end{equation}
Thus it is easy to know that $V_\infty\geq1$. The proof is completed.
\end{proof}

\begin{cor}\label{coro}
If $1<V_\infty\leq\infty$, then for any $k\in(0,\sqrt{V_\infty-1})$,
$$\lim\limits_{r\to\infty}\sup u(r)e^{kr}<\infty.$$
\end{cor}
\begin{proof}
If $1<V_\infty<\infty$, we have $\lim\limits_{r\to\infty}z^2=V_\infty-1$ by applying  \eqref{limz}. Therefore, for any $k\in(0,\sqrt{V_\infty-1})$ and $r$ sufficiently large, we obtain
$-\frac{u^{\prime}}{u}\geq k.$
Integrating it over $(0,r)$, we can obtain $\lim\limits_{r\to\infty}\sup u(r)e^{kr}\leq u_0<\infty$.

If $V_\infty$ is infinite, then $z$ is also unbounded. So there still exist $K>0$ such that $-\frac{u^{\prime}}{u}\geq K$ and then $\lim\limits_{r\to\infty}\sup u(r)e^{Kr}\leq u_0<\infty$. Thus the proof is completed.
\end{proof}

\begin{thm}\label{thel}
The set $\mathcal{G}$ has exactly one element.
\end{thm}
\begin{proof}
We prove it by contradiction. Suppose on the contrary that $u_1(0),u_2(0)\in\mathcal{G}$. Without loss of generality, we assume  $u_2(0)>u_1(0)$. According to  Lemma \ref{lem6}, we have $u_2(r) > u_1(r) >0$ for all $r\geq0$. Since $V_2(0)-V_1(0)=0$, by \eqref{jihao1},  we have
$$(V_2(r)-V_1(r))^{\prime}=\frac{1}{r^{N-1}}\int_0^r(u_2^p(s)-u_1^p(s))s^{N-1}ds>0.$$
and then
\begin{equation}\label{v1v2}
V_2(r)>V_1(r)\mbox{ in }(0,\infty).
\end{equation}
Furthermore,  according to  \eqref{jihao9} for their Wronskian $\omega(r)$, we get $\omega(0)=0$ and
$$(\omega r^{N-1})^{\prime}=(V_2-V_1)u_1u_2r^{N-1}>0,$$
which implies that $\omega(r)r^{N-1}>0$ and then $\omega(r)r^{N-1}$ is a positive and strictly increasing function in $(0,\infty)$.

On the other hand, we claim that
\begin{equation}\label{duanyan3}
\lim\limits_{r\to\infty}\omega (r)r^{N-1}=0.
\end{equation}
Indeed,  since $u_2(0)>u_1(0)>0$,  we have $u_2(r)>u_1(r)\mbox{ on }(0,\infty)$ according to Lemma \ref{lem6}. This together with \eqref{v1v2} and Proposition \ref{prop1}, implies that
\begin{equation}\label{voo}
V_{2}(\infty)>V_{1}(\infty)\geq1.
\end{equation}
By Corollary\ref{coro}, we know for any $k\in(0,\sqrt{V_2(\infty)-1})$,
$$\lim\limits_{r\to\infty}\sup u_2(r)r^{N-1}<\lim\limits_{r\to\infty}\sup u_2(r)e^{kr}<\infty.$$
According to \eqref{jihao3}, we have $V_2(r)\leq\frac{u_2^{p}(0)r^{2}}{2N}$ for any $r\in(0,\infty)$. Then
$$u_2^{\prime}r^{N-1}=\int_0^r(V_2(s)-1)u_2(s)s^{N-1}ds\leq\int_0^r\frac{u_2^p(0)}{2N}u_2(s)s^{N+1}ds. $$
This together with the exponential decay of $u_2$, yields that $u_2^{\prime}r^{N-1}$ and $u_2r^{N-1}$ is uniformly bounded. Notice that  $\lim\limits_{r\to\infty}u(r)=\lim\limits_{r\to\infty}u^{\prime}(r)=0$ and
$$|\omega(r)r^{N-1}|\leq|u_1||u_2^{\prime}r^{N-1}|+|u_1^{\prime}||u_2r^{N-1}|\leq c_1|u_1|+c_2|u_1^{\prime}|,$$
where $c_1,c_2$ are positive constants. The claim holds. Therefore, the proof is finished.
\end{proof}

\end{document}